\documentclass{article}
\usepackage[utf8]{inputenc}
\usepackage{amsmath, amsthm, amsfonts, amssymb, bbm}
\usepackage{geometry}
\usepackage{hyperref}
\usepackage[shortlabels]{enumitem}
\newtheorem{theorem}{Theorem}
\newtheorem{thm}{Theorem}
\newtheorem{lemma}[thm]{Lemma}

\newtheorem{quest}{Question}

\providecommand{\abs}[1]{\left\vert#1\right\vert}

\providecommand{\floor}[1]{\left\lfloor#1\right\rfloor}

\renewcommand{\Pr}{\mathbb{P}}
\newcommand{\pr}{\mathbb{P}}
\newcommand{\E}{\mathbb{E}}
\newcommand{\Var}{\operatorname{Var}}
\newcommand{\eps}{\varepsilon}

\newcommand{\cD}{\mathcal{D}}
\newcommand{\cE}{\mathcal{E}}

\newcommand{\sint}{s}
\newcommand{\rsize}{r}
\newcommand{\uu}{u}

\newcommand\dd{\,\mathrm{d}}

\newcommand\xpar[1]{(#1)}
\newcommand\bigpar[1]{\bigl(#1\bigr)}
\newcommand\Bigpar[1]{\Bigl(#1\Bigr)}
\newcommand\biggpar[1]{\biggl(#1\biggr)}

\newcommand\bigsqpar[1]{\bigl[#1\bigr]}
\newcommand\Bigsqpar[1]{\Bigl[#1\Bigr]}
\newcommand\biggsqpar[1]{\biggl[#1\biggr]}

\newcommand\lrsqpar[1]{\left[#1\right]}

\newcommand\bigcpar[1]{\bigl\{#1\bigr\}}
\newcommand\Bigcpar[1]{\Bigl\{#1\Bigr\}}

\newcommand\Biggcpar[1]{\Biggl\{#1\Biggr\}}

\newcommand{\indic}[1]{\mathbbm{1}_{\{{#1}\}}}

\newcommand\noproof{\qed}

\let\OLDthebibliography\thebibliography
\renewcommand\thebibliography[1]{
  \OLDthebibliography{#1}
  \setlength{\parskip}{0pt}
  \setlength{\itemsep}{0pt plus 0.3ex}
}

\begin{document}
\title{Two-Point Concentration of the Domination Number \\ of Random Graphs}
\author{Tom Bohman\thanks{Department of Mathematical Sciences, Carnegie Mellon University, Pittsburgh, PA 15213, USA. Email: {\tt tbohman@math.cmu.edu}. Supported by NSF grant DMS-2246907 and a grant from the Simons Foundation (587088, TB).}
 \and Lutz Warnke\thanks{Department of Mathematics, University of California, San Diego, La Jolla CA~92093, USA. E-mail: {\tt lwarnke@ucsd.edu}. Supported by NSF~CAREER grant~DMS-2225631 and a Sloan Research Fellowship.}
 \and Emily Zhu\thanks{Department of Mathematics, University of California, San Diego, La Jolla CA~92093, USA. E-mail: {\tt e9zhu@ucsd.edu}. Supported by the NSF Graduate Research Fellowship Program under Grant No.~DGE-2038238.}}
\date{February 12, 2024}
\maketitle

\begin{abstract}
We show that the domination number of the binomial random graph $G_{n,p}$ with edge-probability~$p$ is concentrated on two values for~${p \ge n^{-2/3+\eps}}$, 
and not concentrated on two values for general~${p \le n^{-2/3}}$. 
This refutes a conjecture of Glebov, Liebenau and Szab\'o, who showed two-point concentration for ${p \ge n^{-1/2+\eps}}$, and 
conjectured that two-point concentration fails for ${p \ll n^{-1/2}}$. 

The proof of our main result requires a Poisson type approximation for the probability that a random bipartite graph has no isolated vertices, 
in a regime where standard tools are unavailable (as the expected number of isolated vertices is relatively large). 
We achieve this approximation by adapting the proof of Janson's inequality to this situation, and this adaptation may be of broader interest. 
\end{abstract}

\section{Introduction}
In the theory of random graphs, the study of the extent of concentration of key parameters of the binomial random graph~$G_{n,p}$ has been a central issue from the very beginning~\cite{B2001,JLR2000,FK2016}. 
Landmark two-point concentration results for parameters such as the clique number and chromatic number played an important role for the development of the field~\cite{matula,1976cliques,ss,L1991,ak,AN}, 
and thus one might believe that it is nowadays merely an advanced 
exercise to establish such sharp concentration results for~$G_{n,p}$.
But this is not always the case: some parameters 
exhibit interesting new phenomena or difficulties that require a further development of the theory~\cite{h,hr,bohman2022,SWZ2023}. 
In particular, in this paper two-point concentration requires us to go beyond standard concentration inequalities, i.e., to adapt the proof of Janson's inequality in a tailor-made way. 

We study two-point concentration of the domination number~$\gamma(G_{n,p})$ of the binomial random graph~$G_{n,p}$.
The domination number~$\gamma(G)$ of a graph~$G$ is the smallest number of vertices of a dominating set of~$G$, i.e., a set~$K$ such that every vertex~$v \in V(G) \setminus K$ has at least one neighbor in~$K$. 
This well-known parameter has been studied from many different angles, including graph theory~\cite{Ore1962,BC1979,West1996,HHSS1998}, complexity theory~\cite{GJ1979,ABFKR2002,HHHH2002,AFR2004,DFHT2005}, random graph theory~\cite{W1981,NS1993,WG2001,glebov2015}, and notably also the probabilistic method~\cite{PM}, where in the authoritative textbook by Alon and Spencer it appears as the third example.
For random graphs, the quest for two-point concentration of the domination number~$\gamma(G_{n,p})$ has a long history: 
it was first established around~1981 by Weber~\cite{W1981} for uniform random graphs, where the edge-probability~$p=p(n)$ satisfies~${p=1/2}$. 
This was extended around~2001 by Wieland and Godbole~\cite{WG2001} to fairly dense random graphs~$G_{n,p}$ with~${p \gg \sqrt{(\log \log n)/(\log n)}}$.
Two-point concentration of~$\gamma(G_{n,p})$ was further improved in~2015 to much sparser random graphs~$G_{n,p}$ with~${p \gg (\log n)^2n^{-1/2}}$ by Glebov, Liebenau and Szab\'o~\cite{glebov2015}, 
who also conjectured that their result is essentially best possible, i.e., that two-point concentration of~$\gamma(G_{n,p})$ ought to fail for~$p \ll n^{-1/2}$. 

The main result of this paper shows that two-point concentration of the domination number~$\gamma(G_{n,p})$ of the binomial random graph~$G_{n,p}$ 
with edge-probability~$p=p(n)$ extends down to roughly~$p = n^{-2/3}$, 
refuting the aforementioned conjecture of Glebov, Liebenau and Szab\'o. 
\begin{theorem}[Two-point concentration for~$p \ge n^{-2/3+\eps}$]
\label{thm:main}
If~$p = p(n)$ 
satisfies ${(\log n)^3n^{-2/3} \leq p \leq 1}$, 
then the domination number~$\gamma(G_{n,p})$ of the binomial random graph~$G_{n,p}$ is concentrated on at most two values, 
i.e., there exists~$\hat{\rsize} = \hat{\rsize}(n,p)$ such that $\Pr\bigl(\gamma(G_{n,p}) \in \{\hat{\rsize},\hat{\rsize}+1\}\bigr) \to 1$ as~$n \to \infty$. 
\end{theorem}
\noindent
The range of the edge-probability~$p=p(n)$ in Theorem~\ref{thm:main} is best possible up to logarithmic factors: 
the following anti-concentration result shows that two-point concentration of~$\gamma(G_{n,p})$ fails below~$p = n^{-2/3}$. 
\begin{lemma}[No two-point concentration for~$p \le n^{-2/3}$]\label{lem:anti}
If~$p = p(n)$ 
satisfies ${n^{-1} \ll p \ll (\log n)^{2/3}n^{-2/3}}$, 
then there exists $q = q(n)$ with $p \le q \le 2p$ such that 
the domination number~$\gamma(G_{n,q})$ of~$G_{n,q}$ 
is not concentrated on two values: 
we have $\max_{r \ge 0}\Pr\bigl(\gamma(G_{n,q}) \in \{r,r+1\}\bigr) \le 3/4$ for an infinite sequence of~$n$.  
\end{lemma}
\noindent
A generalization of Lemma~\ref{lem:anti} shows that we cannot strengthen Theorem~\ref{thm:main} to one-point concentration, see Lemma~\ref{lem:anti:general} in Section~\ref{sec:ss}.
We did not optimize the logarithmic factors in the range of~$p=p(n)$ in Theorem~\ref{thm:main}, but it seems that some extra bells and whistles might be needed to potentially match the range of~Lemma~\ref{lem:anti}. 

The proof of the two-point concentration result Theorem~\ref{thm:main} is based on a careful application of the first and second moment method. 
As we shall discuss in Section~\ref{sec:barrier}, 
a major new difficulty starts to appear around~${p = n^{-1/2}}$ in the required variance calculations, 
where we need to establish a Poisson type approximation result in a regime where standard tools are unavailable. 
We achieve this approximation by adapting the proof of Janson's inequality~\cite{JLR1990,SJ1990,RW2015,JS2016} to the situation: this new adaptation idea is the main conceptual contribution of this paper, and we believe that it may be of broader interest (since Janson's inequality is widely used in the study of random discrete~structures);  
see Sections~\ref{sec:novelty} and~\ref{sec:Poisson} for more details. 

The proof of the anti-concentration result Lemma~\ref{lem:anti} is based on a simple coupling argument that was developed by Alon and Krivelevich~\cite{ak} and Sah and Sawhney~\cite{bohman2022} for the clique cover number and independence number of~$G_{n,p}$, respectively. 
Note that the statement of Lemma~\ref{lem:anti} 
does not rule out that~$\gamma(G_{n,p})$ could be two-point concentrated for some edge-probabilities~$p=p(n)$ with $n^{-1} \ll p \ll (\log n)^{2/3}n^{-2/3}$, but this kind of non-monotone concentration behavior would be rather surprising here (though not impossible in general, see~\cite{SW2022}). 
Strengthening this anti-concentration statement 
to all such~$p=p(n)$  
is one of several interesting open questions that we discuss in the concluding Section~\ref{sec:conclusion} of this~paper.

\subsection{Organization of the paper}
In the next subsection we pinpoint a new difficulty that arises in the proof of Theorem~\ref{thm:main} below~${p=n^{-1/2}}$, 
and discuss why it requires us to adapt Janson's inequality in a tailor-made way. 
In Section~\ref{sec:novelty} we illustrate our adaptation of Janson's inequality in a simpler setting, namely in the context of bounding the probability that a random bipartite graph has no isolated vertices. 
The proof of our main two-point concentration result Theorem~\ref{thm:main} follows in Section~3, 
and Section~4 contains the proof of the anti-concentration result Lemma~\ref{lem:anti} (and a generalization thereof). 
We close the paper with a Conclusion that includes some open questions and commentary on the concentration of the domination number of~$G_{n,p}$ below~$p = n^{-2/3}$, see Section~\ref{sec:conclusion}.

\subsection{Poisson approximation difficulty below~$p=n^{-1/2}$}\label{sec:barrier}  
In the following we informally discuss and pinpoint one major difficulty in the proof of Theorem~\ref{thm:main} that only arises below~${p=n^{-1/2}}$. 
The goals here are two-fold: (i)~to illustrate why the previous work of Glebov, Liebenau and Szab\'o~\cite{glebov2015} stopped around~$p=n^{-1/2}$, 
and (ii)~to motivate why an adaption of Janson's inequality is key for establishing the desired two-point concentration $\Pr\bigl(\gamma(G_{n,p}) \in \{\hat{\rsize},\hat{\rsize}+1\}\bigr) \to 1$ as~$n \to \infty$.   

The proof of Theorem~\ref{thm:main} is based on the first and second moment method. 
To this end, let the random variable~$X_r$ denote the number of dominating sets of size~$r$ in $G_{n,p}$. 
Guided by the natural `expectation threshold' heuristic, we choose $\hat{r}$ such that~${\E X_{\hat{r}+1} \to \infty}$ and~${\E X_{\hat{r}-1} \to 0}$ (see~\eqref{def:hatr} and Lemma~\ref{lem:estimates} in Section~\ref{sec:proof}). 
For sufficiently small~$p$, a standard first moment calculation (see Lemma~\ref{lem:estimates} in Section~\ref{sec:proof}) gives
\begin{equation}\label{heur:rhat}
\hat{r} = \frac{ \log(np) - 2 \log\log(np) + o(1)}{p}.
\end{equation}
A central issue for the second moment calculation is to determine the probability that pairs of vertex sets of size $r:=\hat{r}+1$ both dominate the full random graph~$G_{n,p}$. 
More precisely, for most pairs~$A,B$ of sets of size~$|A|=|B|=r$ we intuitively need to show approximate independence of the event~$\cD_A$ that $A$ is a dominating set and the event~$\cD_B$ that $B$ is a dominating set, i.e., that 
\begin{equation}\label{heur:goal:approx}
\Pr(\cD_A \cap \cD_B) \le (1+o(1)) \Pr(\cD_A) \Pr(\cD_B). 
\end{equation}
This estimate requires a key idea that we shall now discuss in the special case\footnote{For~$p \gg (\log n) n^{-1/2}$ the disjoint case~${|A \cap B|=0}$ is in fact the most important one: the crux is that in~$G_{n,p}$ two random vertex-subsets~$A,B$ of size~$r$ then typically intersect in about~$r \cdot r/n \approx \log^2(np)/(np^2) = o(1)$ many vertices.} where the sets~$A$ and~$B$ of size~$r$ are vertex disjoint, i.e., satisfy~${|A \cap B|=0}$. 
To this end we write the event~$\cD_A$ as 
\begin{equation*}
\underbrace{\bigcpar{\text{$A$ is a dominating set}}}_{=\cD_A} = \underbrace{\bigcpar{\text{$A$ dominates $[n] \setminus (A \cup B)$}}}_{=: \cD_{A, [n] \setminus (A \cup B)}} \cap \underbrace{\bigcpar{\text{$A$ dominates $B$}}}_{=: \cD_{A,B}},
\end{equation*}
and analogously~$\cD_B = \cD_{B, [n] \setminus (A \cup B)} \cap \cD_{B,A}$. 
The events~$\cD_{A, [n] \setminus (A \cup B)} \cap \cD_{B, [n] \setminus (A \cup B)}$ and~$\cD_{A,B} \cap \cD_{B,A}$ are independent.
For disjoint~$A$ and~$B$, the events~$\cD_{A, [n] \setminus (A \cup B)}$ and $\cD_{B, [n] \setminus (A \cup B)}$ are also independent.
It follows that our goal~\eqref{heur:goal:approx} is equivalent to approximate independence of the events~$\cD_{A,B}$ and~$\cD_{B,A}$, i.e.,~to 
\begin{equation}\label{heur:goal}
\Pr(\text{$\cD_{A,B} \cap \cD_{B,A}$}) \le (1+o(1)) \Pr(\cD_{A,B})\Pr(\cD_{B,A}) .
\end{equation}
Since~$A$ dominates~$B$ if and only if every vertex~$v \in B \setminus A$ has at least one neighbor inside~$A$, 
using~$|A \cap B|=0$ and the estimate~\eqref{heur:rhat} for~$r=\hat{r}+1$ it follows by a standard random graph calculation~that 
\begin{equation}\label{heur:prob}
\Pr(\cD_{A,B})=\Pr(\cD_{B,A}) = ( 1 - (1- p)^{r})^{r} 
= e^{- r(1- p)^{r} + o(1)} 
= \exp \left\{ -\frac{ \log^3(np)}{ np^2} + o(1)\right\}.
\end{equation}
Note that the probabilities in~\eqref{heur:prob} tend to one when~$p \gg (\log n)^{3/2} n^{-1/2}$, 
in which case~\eqref{heur:goal} simply follows from ${\Pr(\text{$\cD_{A,B} \cap \cD_{B,A}$})} \le 1$. 
Hence the required approximate independence~\eqref{heur:goal} was not a material issue in previous work~{\cite{W1981,WG2001,glebov2015}}, 
where the event~$\cD_{A,B} \cap \cD_{B,A}$ was simply ignored (which is not possible below~$p=n^{-1/2}$).  

To go below~$p =(\log n)^{3/2} n^{-1/2}$ we thus need to carefully estimate~$\Pr(\text{$\cD_{A,B} \cap \cD_{B,A}$})$, 
which turns out to be a major difficulty. 
Recalling that~$A$ and~$B$ are disjoint, 
observe that the event~$\cD_{A,B} \cap \cD_{B,A}$ that~$A$ and~$B$ dominate each other is simply the event 
that~$X=0$, where~$X$ denotes the number of isolated vertices in the induced bipartite subgraph of~$G_{n,p}$ with
parts $A,B$. 
Since the expected number of isolated vertices in this graph equals $\E X = 2r ( 1- p)^{r}$, 
by comparison with~\eqref{heur:prob} it follows that our goal~\eqref{heur:goal} is equivalent~to 
\begin{equation}\label{heur:goal:Janson}
\Pr(X=0) \le (1+o(1)) e^{-\E X} ,
\end{equation}
i.e., to a  Poisson type approximation for the probability that the discussed random bipartite graph with parts~$A$ and~$B$ has no isolated vertices.
Janson's inequality is the standard tool for this kind of Poisson approximation, 
which however is insufficient for establishing~\eqref{heur:goal:Janson}: the issue is that the event that vertex $ a \in A$ is isolated and the event that vertex $b\in B$ is isolated are dependent events, which implies that the `dependency' parameter $\Delta$ in Janson's inequality is too large to be useful. 
Other standard approaches such as the method of moments, Stein–Chen method and inclusion-exclusion type arguments are also unavailable:
the issue is that these usually only work when the expectation~$\E X$ is relatively small, 
whereas in view of~\eqref{heur:prob} the expectation can here be as large as~$\E X = \Theta(\log^3(np)/(np^2))= n^{1/3+o(1)}$ when~$p = n^{-2/3+o(1)}$. 

In this paper we overcome this difficulty by adapting the proof of Janson's inequality to leverage the fact that, while the event that $a \in A$ is isolated and the event that $ b\in B$ is isolated are dependent events, they are only very mildly dependent. As we believe that this adaptation of Janson's inequality may be of
broader interest, in Section~\ref{sec:novelty} we state a separate Theorem and proof that illustrate the argument in the context of random bipartite graphs. 
Of course, the proof of Theorem~\ref{thm:main} requires that we consider the probability that arbitrary pairs~$A,B$ of vertex sets of size $r$ dominate each other (rather than just disjoint pairs of sets):  this extension to all pairs~$A,B$ is treated in~\eqref{eq:sintbig}--\eqref{eq:rhoseq} and Section~\ref{sec:Poisson}, as part of the proof of Theorem~\ref{thm:main}.

\section{Key ingredient: Poisson approximation beyond Janson}\label{sec:novelty}
As discussed in Section~\ref{sec:barrier} of the introduction, 
the proof of our main result Theorem~\ref{thm:main} requires an 
exponential upper bound for the probability that a certain sum~$X$ of mildly dependent 
indicator random variables is zero, that does not seem to follow easily from known
methods.
In particular, for moderately large expectations~$\E X$ 
we need a Poisson approximation result of the~form
\begin{equation}\label{def:Poisson}
\Pr(X=0) \le (1+o(1))e^{-\mu} \quad \text{ with } \quad \mu := \E X ,
\end{equation}
but it turns out that standard tools and approaches fail in our application. 
Indeed, Janson's inequality gives insufficient bounds, as we shall discuss.  
Furthermore, the method of moments, the Stein–Chen method and include-exclusion arguments are all unavailable, 
since these usually only manage to establish~\eqref{def:Poisson} when the expectation is fairly small, say~${\mu=\Theta(1)}$ or perhaps~${\mu=o(\log n)}$. 
For our application we thus had to develop a new argument for establishing Poisson approximation of form~\eqref{def:Poisson} that takes advantage of the fact the dependence among the indicator random
variables is mild, 
and in this section we illustrate our argument.

Let~$X$ denote the number of isolated vertices in the bipartite random graph~$G_{N,N,p}$ with vertex classes~$V_1$ and~$V_2$ of size~$|V_1|=|V_2|=N$.
Writing~$V := V_1 \cup V_2$ we can then decompose 
\begin{equation}\label{def:X}
X = \sum_{v \in V}I_v \quad \text{ with } \quad \mu = \E X = \sum_{v \in V} \E I_v ,
\end{equation}
where~$I_v$ is the indicator random variable for the event that~$v$ is isolated in~$G_{N,N,p}$. 
Janson's inequality yields 
\begin{equation}\label{eq:Janson}
\Pr(X=0) \le e^{-\mu + \Delta/2} \quad \text{ with } \quad  \Delta := \sum_{v \in V}\sum_{w \in V: w \sim v} \E(I_vI_w) ,  
\end{equation}
where we write~$w \sim v$ when the vertices~$w$ and~$v$ can potentially be neighbors in~$G_{N,N,p}$ (which means that they are in different vertex classes). 
Janson's inequality~\eqref{eq:Janson} fails to establish the desired Poisson approximation~\eqref{def:Poisson} because~$\Delta= \mu^2/[2(1-p)]$ is too~large. 
Formally~\eqref{eq:Janson} gives a poor bound because every variable~$I_v$ has some dependency with half of all variables (all variables~$I_w$ with~$w \sim v$), 
but the main conceptual reason is that Janson's inequality is unable to take into account that all these dependencies are fairly mild. 

We obtain the following improved inequality
by taking the discussed mild dependencies into account,
which effectively allows us to reduce the~$\Delta$ term in the exponent of Janson's inequality~\eqref{eq:Janson} to about~$\Delta p$. 
\begin{theorem}\label{thm:Poisson} 
Let $X$ denote the number of isolated vertices in the random bipartite graph $G_{N,N,p}$, and let
$\mu$ and~$\Delta$ be defined as in~\eqref{def:X} and~\eqref{eq:Janson} above. Then we have
\begin{equation}\label{eq:Poisson}
e^{-\mu+\mu^2/(2N)} \le \Pr(X=0) \le e^{-\mu + \Delta p \cdot \log(1+\mu/(\Delta p))} ,
\end{equation}
where the lower bound holds when~$\Pr(I_v=1) \le 1/2$. 
\end{theorem}
\noindent
It is instructive that~$\Delta p$ naturally appears in the variance ${\Var X = \mu-\mu^2/(2N)  + \Delta p}$. 
In particular, ignoring logarithmic factors in the constraints, 
\mbox{Theorem~\ref{thm:Poisson}} gives the following Poisson approximation results: 
${\Pr(X=0)=e^{-(1+o(1))\mu}}$ when ${|\Var X-\mu| = o(\mu)}$, 
and ${\Pr(X=0)=(1+o(1))e^{-\mu}}$ when ${|\Var X-\mu| = o(1)}$.

Our proof of Theorem~\ref{thm:Poisson} uses correlation inequalities as well as conditional independence arguments, 
and also exploits the bipartite structure of~$G_{N,N,p}$ in a crucial way. 
After giving the moment generating function~${\Phi(s) :=  \E(e^{-sX})}$ related basic setup, 
we shall below describe some of our new ideas for establishing the upper bound in inequality~\eqref{eq:Poisson}, 
and discuss how they differ from the proof of Janson's inequality~\eqref{eq:Janson}. 
Indeed, for any~$\lambda \ge 0$ we have~$\Phi(\lambda) = \E(e^{-\lambda X}) \ge \Pr(X=0)$,
which together with~$\Phi(0)=1$ yields
\begin{equation}\label{eq:zero}
-\log \Pr(X=0) \ge -\log \Phi(\lambda) = -\int_{0}^{\lambda} \bigpar{\log \Phi(s) }' \dd s  .
\end{equation}
Direct calculation shows that the we can write the (negative) logarithmic derivative~as
\begin{equation}\label{eq:MGF}
-\bigpar{\log \Phi(s) }' = \frac{-\Phi'(s)}{\Phi(s)}
= \sum_{v \in V}\frac{\E(I_v e^{-sX})}{\E(e^{-sX})} .
\end{equation}
To establish an upper bound on~$\Pr(X=0)$, 
as usual for concentration inequality proofs, it remains to prove a lower bound on~$-\bigpar{\log \Phi(s) }'$, 
insert that bound into~\eqref{eq:zero}, 
and then minimize the resulting expression with respect to~$\lambda \ge 0$. 
Here the major technical obstacle is that the summands of~$X=\sum_{v \in V}I_v$ are not all independent, 
which in particular means that one cannot factorize~$\E(I_v e^{-sX})$ or~$\E(e^{-sX})$ easily. 

The proof of Janson's inequality~\eqref{eq:Janson} beautifully gets around this obstacle as follows.
First~${X={X'_{v}+Y'_{v}}}$ is decomposed into two parts, where~$X'_v={\sum_{w \in V \setminus \{v\}:w \not\sim v}}I_w$ is the sum of all the variables~$I_w$ that are independent of~$I_v$, and~$Y'_v={I_v + \sum_{w \in V:w \sim v}}I_w$ is the sum over all other~$I_w$. 
Then the FKG inequality and conditional independence are used to derive a lower bound on the numerator~${\E(I_v e^{-sX})}={\Pr(I_v=1)\E(e^{-sX'_{v}} e^{-sY'_{v}}|I_v=1)}$ that cancels out the~${\E(e^{-sX})}$ in the denominator of~\eqref{eq:MGF},
followed by two applications of Jensen's inequality to obtain a useful lower bound on the right-hand size of~\eqref{eq:MGF}; see~\cite[Theorem~1]{SJ1990} or~\cite[Theorem~14 and~18]{JLR2000} for the proof~details. 

Our proof of Theorem~\ref{thm:Poisson} takes a related but conceptually different approach to recover the crucial extra factor~$p$ in the exponent of~\eqref{eq:Poisson} that is 
missing in Janson's inequality~\eqref{eq:Janson}.   
The key idea is to decompose~${X=I_v+X_{v}-Y_{v}}$ in a tailor-made way, which carefully takes into account that there are only very mild dependencies between the summands of~$X= \sum_{v \in V}I_v$. 
As we shall see, this enables us to derive an exact expression for the numerator~$\E(I_v e^{-sX})={\Pr(I_v=1)e^{-s}\E(e^{-sX_{v}})}$, and then use the FKG inequality and conditional independence to derive an upper bound on the denominator~${\E(e^{-sX})} \le {\E(e^{-sX_v} e^{sY_v})}$ that cancels out the~$\E(e^{-sX_{v}})$ in the numerator. 
This way we obtain a lower bound on the right-hand size of~\eqref{eq:MGF} that eventually allows us to establish the improved upper bound~in~\eqref{eq:Poisson}. 

\begin{proof}[Proof of Theorem~\ref{thm:Poisson}]
We start with the upper bound in inequality~\eqref{eq:Poisson}, proceeding from the discussed identities~\eqref{eq:zero} and~\eqref{eq:MGF}. 
Fix a vertex~$v \in V$. 
For any vertex~$w \in V \setminus \{v\}$ the key idea is to introduce~$I^{v}_{w}$, 
which is the indicator random variable for the event that~$w$ is isolated when we ignore the edge-status of the pair~$vw$ (i.e., it does not matter whether the pair~$vw$ is an edge or not). 
Writing~$\Gamma(v)$ for the neighbors of~$v$ in~$G_{N,N,p}$, 
note that~${I_{w} = I^{v}_{w}\indic{w \not\in \Gamma(v)}}$ for any vertex~$w \in V \setminus \{v\}$ 
(since we deterministically have~$\indic{w \not\in \Gamma(v)}=1$ when~$v$ and~$w$ are in the same vertex class). 
Since by definition~$v \not\in\Gamma(v)$, we then decompose~${X= \sum_{v \in V}I_v}$~as
\begin{equation}\label{eq:X:decomp}
\begin{split}
X \; = \; I_v \: + \sum_{w \in V \setminus \{v\}}I^{v}_{w}\indic{w \not\in \Gamma(v)} \; = \; I_v \: + \: \underbrace{\sum_{w \in V \setminus \{v\}}I^{v}_{w}}_{=: X_v} \: - \: \underbrace{\sum_{w \in \Gamma(v)}I^{v}_{w}}_{=: Y_v} . 
\end{split}
\end{equation}
Note that~$I_v=1$ corresponds to the event that~$v$ has no neighbors in~$G_{N,N,p}$, i.e., that~${\Gamma(v)=\emptyset}$. 
Hence~$I_v=1$ implies~$Y_v=0$, so using the decomposition~$X = I_v + X_v - Y_v$ from~\eqref{eq:X:decomp} it follows that
\begin{equation*}
 \E(I_v e^{-sX}) = \Pr(I_v=1)\E(e^{-sX} \mid I_v=1)
 = \Pr(I_v=1)e^{-s} \E(e^{-sX_v} \mid I_v=1).
\end{equation*}
Since~$X_v$ is by construction independent of the edge-status of all vertex pairs adjacent to~$v$, 
it follows~that 
\begin{equation}\label{eq:IvMGF:eq}
 \E(I_v e^{-sX}) = \Pr(I_v=1) e^{-s} \E(e^{-sX_v}).
\end{equation}

With an eye on~\eqref{eq:MGF}, we next relate~$\E(e^{-sX})$ with~$\E(e^{-sX_v})$. 
Using that~\eqref{eq:X:decomp} implies~$X \ge X_v-Y_v$, by taking all possible neighborhoods~$\Gamma(v)$ of~$v$ into account via conditional expectations, we obtain~that
\begin{equation}\label{eq:MGF:1}
 \E(e^{-sX}) \le  \E\bigpar{e^{-sX_v} e^{sY_v}}  = \E \Bigsqpar{\E\bigpar{e^{-sX_v} e^{sY_v} \: \big| \: \Gamma(v)}}.
\end{equation}
Note that conditional on~$\Gamma(v)=S$ we have a new product space, 
where each of the remaining $N(N-1)$ potential pairs~$V_1 \setminus\{v\} \times V_2 \setminus \{v\}$ of~$G_{N,N,p}$ is included independently with probability~$p$. 
Furthermore, in this product space we have~$Y_v=\sum_{w \in S}I^{v}_{w}$, i.e., the range of summation in~$Y_v$ is deterministic. 
Since~$e^{-sX_v}$ is a increasing function and~$e^{sY_v}$ is a decreasing decreasing function (of the edge-status of the remaining $N(N-1)$ potential pairs), 
by the FKG inequality and the discussed independence of~$X_v$ and~$\Gamma(v)$ it follows~that 
\begin{equation}\label{eq:MGF:2}
\begin{split}
\E\bigpar{e^{-sX_v} e^{sY_v} \: \big| \:  \Gamma(v)=S} & \le \E\bigpar{e^{-sX_v} \: \big| \:  \Gamma(v)=S} \E\bigpar{e^{sY_v} \: \big| \: \Gamma(v)=S} \\
& = \E\bigpar{e^{-sX_v}} \E\bigpar{e^{sY_v} \: \big| \: \Gamma(v)=S} .
\end{split}
\end{equation}
The crux is now that, due to the bipartite structure of~$G_{N,N,p}$, all summands~$I^{v}_{w}$ of ${Y_v=\sum_{w \in S}I^{v}_{w}}$ are independent random variables (as all the vertices~$w \in S=\Gamma(v)$ are in the same vertex class of~$G_{N,N,p}$). 
Since each~$I^{v}_{w}$ is independent of~$\Gamma(v)$, 
writing~$\Pi_w:=\Pr(I^{v}_{w}=1)$ it follows~that 
\begin{equation}\label{eq:MGF:2b}
\begin{split}
\E\bigpar{e^{sY_v} \: \big| \: \Gamma(v)=S} & = \E\bigpar{\prod_{w \in S}e^{sI^{v}_{w}} \: \big| \: \Gamma(v)=S} = \prod_{w \in S} \E\bigpar{e^{sI^{v}_{w}} \: \big| \: \Gamma(v)=S}\\
& = \prod_{w \in S} \E\bigpar{e^{sI^{v}_{w}}} = \prod_{w \in S}\bigpar{1+(e^{s}-1)\Pi_w} \\
& = \underbrace{\prod_{w \in V: w \sim v}\bigpar{1+(e^{s}-1)\Pi_w}^{\indic{vw \in E(G_{N,N,p})}}}_{=: Z_v} ,
\end{split}
\end{equation}
where for the last equality we used that only vertices~$w \in V$ with~$w \sim v$ can potentially be neighbors of~$v$ in~$G_{N,N,p}$. 
Since each edge of~$G_{N,N,p}$ is included independently with probability~$p$,  it follows~that 
\begin{equation}\label{eq:MGF:3}
\begin{split}
\E(Z_v) & = \prod_{w \in V: w \sim v} \E \Bigsqpar{\bigpar{1+(e^{s}-1)\Pi_w}^{\indic{vw \in E(G_{N,N,p})}}} \\
& = \prod_{w \in V: w \sim v}\Bigsqpar{1+(e^{s}-1)\Pi_w p} . 
\end{split}
\end{equation}
Combining first~\eqref{eq:MGF:1} with~\eqref{eq:MGF:2}--\eqref{eq:MGF:2b}, and then~\eqref{eq:MGF:3} with the well-known inequality~$1+x \le e^x$, we obtain~that
\begin{equation}\label{eq:IvMGF:upper}
\E(e^{-sX}) \le \E\bigpar{e^{-sX_v}} \E(Z_v) \le \E\bigpar{e^{-sX_v}} e^{(e^{s}-1)\delta_v p},
\end{equation}
where, in view of~$\Pi_w=\Pr(I^{v}_{w}=1)=(1-p)^{N-1}$, $\mu = 2N(1-p)^N$ and~${\Delta = 2N^2(1-p)^{2N-1}}$, we have 
\begin{equation}\label{def:deltav}
\delta_v := \sum_{w \in V: w \sim v}\Pi_w = N(1-p)^{N-1} = \Delta/\mu .
\end{equation}

To sum up, after inserting~\eqref{eq:IvMGF:eq} and~\eqref{eq:IvMGF:upper} into~\eqref{eq:MGF}, 
using~\eqref{def:deltav} and~$\mu=\sum_{v \in V}\Pr(I_v=1)$ we obtain~that  
\begin{equation}\label{eq:MGF:lower}
\begin{split}
-\bigpar{\log \Phi(s) }' & \ge \sum_{v \in V} \Pr(I_v=1) e^{-s}  e^{-(e^{s}-1)\delta_v p} = \mu \cdot e^{-s} e^{-(e^{s}-1) \sigma} , 
\end{split}
\end{equation}  
where~$\sigma := \Delta p/\mu$. 
Integrating the right-hand side of~\eqref{eq:MGF:lower} is non-trivial, 
but it turns out that we obtain a tractable integral if we use the well-known inequality~$e^{-x} \ge 1-x$ to estimate~$e^{-(e^{s}-1) \sigma}$ from below. 
Namely, after inserting the estimate~\eqref{eq:MGF:lower} into inequality~\eqref{eq:zero}, 
using~$e^{-(e^{s}-1)\sigma} \ge 1-(e^{s}-1)\sigma$ it follows~that 
\begin{equation}\label{eq:zero:2}
\begin{split}
-\log \Pr(X=0) 
& \ge \mu \int_{0}^{\lambda} \Bigpar{(1+\sigma)e^{-s} - \sigma} \dd s = \mu \Bigpar{(1+\sigma)(1-e^{-\lambda})- \sigma \lambda} .
\end{split}
\end{equation}
The right-hand side of~\eqref{eq:zero:2} is maximized by choosing~$\lambda = \log\bigpar{1+1/\sigma}$, 
which in view of the identities~${(1+\sigma)(1-e^{-\lambda})}=1$ and~$\sigma = \Delta p/\mu$ 
then readily establishes the upper bound in~\eqref{eq:Poisson}.

The remaining lower bound in~\eqref{eq:Poisson} is well-known and straightforward. 
Indeed, since the indicators~$I_v$ are decreasing functions (of the edge-status of all~$N^2$ potential pairs in~$G_{N,N,p}$),  
by applying the FKG inequality and the well-known inequality~${1-x \ge e^{-x-x^2}}$ for~${0 \le x \le 1/2}$, 
using~${\pi_v :=\Pr(I_v=1)} \le 1/2$ it follows~that 
\[
\Pr(X=0) =\Pr\Bigpar{\bigcap_{v \in V}\{I_v=0\}} \ge \prod_{v \in V} \Pr\xpar{I_v=0}  = \prod_{v \in V}(1-\pi_v) \ge e^{-\sum_{v \in V}\pi_v(1+\pi_v)} .
\]
Noting that~$\mu=\sum_{v \in V}\pi_v$ and~$\pi_v= (1-p)^N=\mu/(2N)$ then completes the proof of inequality~\eqref{eq:Poisson}. 
\end{proof}

The estimates in the above proof are optimized for the Poisson-like case~$\sigma =\Delta p/\mu \to 0$ relevant in our later application in Section~\ref{sec:Poisson}. 
As pointed out by Svante Janson, in the complementary case~$\sigma \to \infty$ our upper bound~\eqref{eq:Poisson} yields $\Pr(X=0) \le e^{-(1+o(1)) \mu/(2\sigma)}$, 
whereas a more careful integration (using the substitution~${t=e^s-1}$ and inequality~${(t+1)^{-2} \ge 1-2t}$) of estimate~\eqref{eq:MGF:lower} yields $\Pr(X=0) \le e^{-(1+o(1)) \mu/\sigma}$.

\section{Two-point concentration: Proof of Theorem~\ref{thm:main}}\label{sec:proof}
This section is devoted to the proof of our main result Theorem~\ref{thm:main}, 
which establishes two-point concentration of $\gamma(G_{n,p})$ for~$(\log n)^{3}(\log n)^{2/3} \le p \le 1$ 
using the first and second moment method. 
For mathematical convenience, we henceforth restrict our attention to edge-probabilities~$p=p(n)$ with~$(\log n)^3n^{-2/3} \leq p \leq (\log n)^{-2}$. 
This restriction is justified, since the earlier work~\cite{WG2001,glebov2015} already covers\footnote{Two point concentration of~$\gamma(G_{n,p})$ for~$p \ge (\log n)^{-2}$ can be deduced from earlier work as follows:
for~$p \to 1$ from~\cite[Theorem~1.1 and Section~3.3]{glebov2015}, 
for constant~$p \in (0,1)$ from~\cite[Theorem~4]{WG2001}, 
and for~$(\log n)^{-2} \le p \ll 1$ from~\cite[Theorem~1.1]{glebov2015}.} 
the case~$p \geq (\log n)^{-2}$. 
To avoid clutter, we henceforth also tacitly assume that~$n$ is sufficiently large whenever necessary (as usual). 

We now introduce the basic proof setup. 
Let~$X_\rsize$ denote the random variable counting the number of dominating sets of size~$\rsize$. 
To avoid clutter, by exploiting symmetry we introduce the shorthand
\begin{equation}
\label{def:tau}
\tau := \pr(A \text{ is a dominating set}) \qquad \text{ for any~$A\in \tbinom{[n]}{\rsize}$.}
\end{equation}
Since~$A$ is a dominating set if and only if every vertex outside of~$A$ has at least one neighbor inside~$A$, 
it follows that the domination probability equals
\begin{equation}
\label{eq:tau}
\tau := \bigpar{1-(1-p)^\rsize}^{n-\rsize} .
\end{equation}
Using linearity of expectation, we infer that the expected number of dominating sets of size~$\rsize$ equals
\begin{equation}
\label{eq:EXr}
\E X_{\rsize} = \sum_{A \in \binom{n}{\rsize}} \tau = \binom{n}{\rsize} \cdot \bigpar{1-(1-p)^\rsize}^{n-\rsize}.
\end{equation}
In order to reuse some standard estimates from previous work~\cite{glebov2015}, it will be convenient to~introduce 
\begin{equation}\label{def:hatr}
\hat{\rsize}=\hat{\rsize}(n,p) := \min\bigcpar{\rsize \: : \: \E X_\rsize \geq 1/(np)},
\end{equation}
where the minimum is taken over all integers~$r$ (in~\cite{glebov2015} the parameter~$\hat{\rsize}$ is shifted by~$-1$). 
Under the discussed extra assumption~$p \leq (\log n)^{-2}$ 
we have~${1/\log(1-p)}={p^{-1}\bigpar{1+O(p)}} = {p^{-1}\bigpar{1+o(1/\log n)}}$, 
which together with Observation~2.1 and Lemma~2.4 from~\cite{glebov2015} readily establishes the following key properties of~$\hat{\rsize}$. 
\begin{lemma}[Properties of~$\hat{\rsize}$ from~\cite{glebov2015}]\label{lem:estimates}
For any integer~$r$ with~$|r-\hat{\rsize}|=O(1)$ we have
\begin{equation}
\label{eq:rsize}
r = \log_{1/(1-p)}\left(\frac{np}{\log^2(np)}(1+o(1))\right) = \frac{\log(np)-2\log\log(np)+o(1)}{p} .
\end{equation}
Furthermore, $\E X_{\hat{\rsize}-1} \to 0$ and~$\E X_{\hat{\rsize}+1} \to \infty$. \noproof
\end{lemma}
\noindent
Note that the lower bound $ \gamma( G_{n,p}) \ge \hat{r}$ whp (with high probability, i.e., with probability tending to one as~$n \to \infty$) follows immediately from Lemma~\ref{lem:estimates}.
Indeed, using Markov's inequality and Lemma~\ref{lem:estimates} we~obtain
\begin{equation}
\pr(\gamma(G_{n,p}) \leq \hat{\rsize}-1) = \pr(X_{\hat{\rsize}-1} \geq 1) \leq \E X_{\hat{\rsize}-1} \to 0 ,
\end{equation}
establishing that whp~$\gamma(G_{n,p}) \geq \hat{\rsize}$, as desired.

It remains to prove that whp~$\gamma(G_{n,p}) \leq \hat{\rsize}+1 =:\rsize$. In the remainder of
this section we prove this using the second moment method.
Note that $\E X_\rsize \gg 1$ by Lemma~\ref{lem:estimates}. 
Using Chebychev's inequality it then follows~that
\begin{equation}
\pr(\gamma(G_{n,p}) \ge \rsize+1) \le \Pr(X_\rsize=0) \le \Pr(|X_\rsize- \E X_\rsize| \ge \E X_\rsize) \le \frac{\Var X_\rsize}{\E[X_\rsize]^2}.
\end{equation} 
To complete the proof of Theorem~\ref{thm:main}, it thus remains to prove that~$\Var X_\rsize  = o(\E[X_\rsize]^2)$. 
Here we organize our variance estimate differently from previous work~\cite{W1981,WG2001,glebov2015}, in order to exploit cancellation effectively. 
Note that two random vertex-subsets $A, B \in \binom{[n]}{\rsize}$ have a typical intersection size of around~$\rsize^2/n$.
Motivated by this we introduce the cutoff parameter 
\begin{equation}
\label{eq:r0}
\rsize_0 := \floor{\rsize^2\log(np)/n},
\end{equation}
where the form~\eqref{eq:rsize} of~$r$ shows that~$\rsize_0 = O(\log^3(np)/(np^2))$. 
Hence~$\rsize_0=0$ when~$p \gg \log(n)^{3/2}n^{-1/2}$. 
Similarly to the domination probability~$\tau$ from~\eqref{def:tau}, 
by exploiting symmetry we introduce the~shorthand
\begin{equation}
\label{eq:rho}
\rho(\sint) := \pr(A,B \text{ are dominating sets}) \qquad \text{ for any~$A,B\in \tbinom{[n]}{\rsize}$ with $\abs{A \cap B} = \sint$,}
\end{equation}
where we shall henceforth always write~$\sint$ for the intersection size of two potential dominating sets. 
With foresight we then decompose the variance into two parts, depending on the size of the intersections:
\begin{equation}
\label{eq:var}
{\Var X_\rsize} 
= \sum_{A,B \in \binom{[n]}{\rsize}} \bigsqpar{\rho(\sint)-\tau^2} \leq \underbrace{\sum_{\substack{A,B\in \binom{[n]}{\rsize}:\\\abs{A \cap B} \leq \rsize_0}} \bigsqpar{\rho(\sint)-\tau^2}}_{=:\mathbb{V}_1} + \underbrace{\sum_{\substack{A,B\in \binom{[n]}{\rsize}:\\\abs{A \cap B} > \rsize_0}} \rho(\sint)}_{=:\mathbb{V}_2}.
\end{equation}

To analyze the variance $\Var X_\rsize$, good upper bounds on the probability~$\rho(\sint)$ are therefore key, since~$\tau$ is known explicitly by~\eqref{eq:tau}. 
In all earlier work~\cite{W1981,WG2001,glebov2015} the following simple upper bound on~$\rho(\sint)$ was used:
\begin{equation}\label{eq:sintbig}
\begin{split}
\rho(\sint) &\le \pr\bigpar{\text{$A$ and $B$ both dominate $[n] \setminus (A \cup B)$}},\\
&= \pr\bigpar{\text{every vertex~$v \in [n] \setminus (A \cup B)$ has a neighbor in both~$A$ and~$B$}}\\
& = \bigpar{1-2(1-p)^\rsize+(1-p)^{2\rsize-\sint}}^{n-2\rsize+\sint}.
\end{split}
\end{equation}
Our main innovation is to analyze the probability~$\rho(\sint)$ more carefully, starting with the observation that
\begin{equation}\label{eq:rhoseq}
\begin{split}
\rho(\sint) &= \pr\bigpar{\text{$A$ and $B$ both dominate $[n] \setminus (A \cup B)$}} \cdot \pr\bigpar{\text{$A$ and $B$ dominate each other}}.
\end{split}
\end{equation}
The extra term $\pr(A,B\text{ dominate each other})$ was ignored in all earlier work~{\cite{W1981,WG2001,glebov2015}}: 
it is not relevant for the edge-probabilities~$p=p(n)$ treated in those works, as discussed in Section~\ref{sec:barrier} of the Introduction. 
Our new Poisson approximation argument from Section~\ref{sec:novelty} is crucial for analyzing this extra term: 
in Section~\ref{sec:Poisson} it allows us to prove the following key~result.
\begin{lemma}[Main technical result]
\label{lem:Poisson}
For any $A,B \in \binom{[n]}{\rsize}$ with ${\abs{A \cap B} = \sint}$, 
\begin{equation}
\label{eq:rho:domeachother}
\pr\bigpar{\text{$A$ and $B$ dominate each other}} \leq (1+o(1)) \cdot (1-2(1-p)^\rsize)^{\rsize-\sint} .
\end{equation}
\end{lemma}
\noindent
By combining the estimate~\eqref{eq:rhoseq} with~\eqref{eq:sintbig} and~\eqref{eq:rho:domeachother} 
we then obtain the following improved upper bound on~$\rho(s)$, 
the crux being that it involves power~$n-\rsize$ instead of~$n-2\rsize+\sint$ as~before:
\begin{equation}
\label{eq:rho:improved}
\rho(s) \le (1+o(1)) \cdot \bigpar{1-2(1-p)^\rsize+(1-p)^{2\rsize-\sint}}^{n-\rsize} .
\end{equation}

For intersections of size~$s=\abs{A \cap B} \le \rsize_0$ and $p \ge (\log n)^3 n^{-2/3}$, 
we show in Section~\ref{sec:v1bound} that
the improved estimate~\eqref{eq:rho:improved} is sufficient to
establish approximate independence of the event that $A$ dominates $B$ and the event that $B$ dominates $A$.
\begin{lemma}[Asymptotic Independence]\label{lem:v1bound}
For~$\rsize_0$ as defined in~\eqref{eq:r0} above, for any~${0 \le \sint \leq \rsize_0}$ we have 
\begin{equation}
\label{eq:cancelbound}
\rho(\sint) \le (1+o(1)) \tau^2 .
\end{equation}
\end{lemma}
\noindent
This result enables cancellation in the second moment calculation 
for much smaller edge-probabilities ${p=p(n)}$ than 
in earlier~work~\cite{W1981,WG2001,glebov2015}, i.e., below~${p=n^{-1/2}}$. 
In particular, using estimates~\eqref{eq:cancelbound} and~\eqref{eq:EXr} it follows~that
\begin{equation}
\label{eq:smallr}
\mathbb{V}_1 = \sum_{\substack{A,B\in \binom{[n]}{\rsize}:\\ \abs{A \cap B} \leq \rsize_0}}\bigsqpar{\rho(\sint)-\tau^2}\leq \sum_{\substack{A,B\in \binom{[n]}{\rsize}:\\ \abs{A \cap B} \leq \rsize_0}} o(\tau^2) = o(\E[X_\rsize]^2) .
\end{equation}

For intersections of size~$s=\abs{A \cap B} > \rsize_0$, we simply write
\begin{equation}
\label{eq:v2}
\frac{\mathbb{V}_2}{\E[X_\rsize]^2} = \sum_{\rsize_0 < \sint \le \rsize}\sum_{\substack{A,B\in \binom{[n]}{\rsize}:\\\abs{A \cap B} = \sint}} \frac{\rho(\sint)}{\E[X_\rsize]^2} = \sum_{\rsize_0 < \sint \le \rsize} \underbrace{\frac{\binom{n}{\rsize}\binom{\rsize}{\sint}\binom{n-\rsize}{\rsize-\sint}}{\binom{n}{\rsize}\binom{n}{\rsize}} \cdot \frac{\rho(\sint)}{\tau^2}}_{=:\uu_\sint}.
\end{equation}
Here the simple upper bound~\eqref{eq:sintbig} on the probability~$\rho(\sint)$ would suffice, 
but the improved upper bound~\eqref{eq:rho:improved} leads to some simplifications in the calculations. 
In Section~\ref{sec:v2bound} we shall thus use~\eqref{eq:rho:improved} and standard estimates for binomial coefficients (which intuitively encapsulate that only few pairs~$A,B$ have rather large~intersections) to obtain the following result.
\begin{lemma}[Counting Estimate]\label{lem:v2bound}
For~$\uu_\sint$ and~$\rsize_0$ as defined in~\eqref{eq:r0} and~\eqref{eq:v2} above, 
\begin{equation}
\label{eq:usbound}
\sum_{\rsize_0 < \sint \le \rsize} \uu_\sint = o(1).
\end{equation}
\end{lemma}
\noindent
Applying estimates~\eqref{eq:v2} and~\eqref{eq:usbound} 
we have
\begin{equation}
\label{eq:bigr}
\mathbb{V}_2 = \E[X_\rsize]^2 \cdot \sum_{\rsize_0 < \sint \le \rsize} \uu_\sint = o(\E[X_\rsize]^2).
\end{equation}

To sum up: from~\eqref{eq:smallr} and~\eqref{eq:bigr} it follows that $\Var X_\rsize \le \mathbb{V}_1+\mathbb{V}_2 =  o(\E[X_\rsize]^2)$, as desired.
In order to complete the proof of Theorem~\ref{thm:main}, it thus remains to prove Lemmas~\ref{lem:Poisson}--\ref{lem:v2bound} in the upcoming Section~\ref{sec:deferred}.

\subsection{Deferred proofs and estimates}\label{sec:deferred}
In this section we give the deferred proofs of Lemmas~\ref{lem:Poisson}--\ref{lem:v2bound}. 
Note that all these results use~$\rsize=\hat{\rsize}+1$, where~$\hat{\rsize}=\hat{\rsize}(n,p)$ is as defined in~\eqref{def:hatr}. 
Hence estimate~\eqref{eq:rsize} applies to~$\rsize$, 
from which it follows~that 
\begin{equation}\label{eq:prsize}
(1-p)^\rsize = (1+o(1))\frac{\log^2(np)}{np} = (1+o(1))\frac{\rsize\log(np)}{n}.
\end{equation}
In the upcoming proofs we shall also use the following well-known binomial approximation result: 
\begin{equation}\label{eq:binomapprox}
(1-p)^a = 1-ap(1+o(1))  \qquad \text{for any integer~$a$ with~$|a| = o(1/p)$.}
\end{equation}

\subsubsection{Poisson approximation: Proof of Lemma~\ref{lem:Poisson} }\label{sec:Poisson}
In this subsection we prove Lemma~\ref{lem:Poisson}, which is our main technical result. 
To establish the claimed bound~\eqref{eq:rho:domeachother} for all intersections sizes~$0 \le \sint \le \rsize$, we need to estimate the probability that~$A$ and $B$ dominate each other when~$|A|=|B|=r$ and~$|A \cap B|=s$.
We first note that~\eqref{eq:rho:domeachother} holds trivially when~$\sint=\rsize$, 
so we may henceforth assume that~$0 \le s \le \rsize-1$. 
Let~$\cE_x$ with~$x \in A \setminus B$ denote the event that~$x$ has no neighbor in~$B$, and let~$\cE_y$ with~$y \in B \setminus A$ denote the event that~$y$ has no neighbor in~$A$. 
Note that $A$ and~$B$ dominate each other if and only if none of the events~$\cE_v$ with~$v \in A \triangle B= (A \setminus B) \cup (B \setminus A)$ hold. 
Defining~$I_v$ as the indicator random variable for the even that~$\cE_v$ holds, we thus obtain~that
\begin{equation}
\label{eq:rho:domeachother:rewritten}
\pr\bigpar{\text{$A$ and $B$ dominate each other}} =\pr(X=0) \qquad \text{ for } \qquad X:=\sum_{v \in A \triangle B}I_v .
\end{equation}
It is instructive to note that in the special case~$s=0$ where~$A$ and~$B$ are disjoint, the random variable~$X$ is exactly the random variable discussed in Section~\ref{sec:novelty} with~$N=r$ and~$V=A \triangle B$.
In the following we complete the proof of Lemma~\ref{lem:Poisson} by observing that it is straightforward to 
generalize the argument given in Section~\ref{sec:novelty} to account for arbitrary intersections sizes~${s = |A \cap B|}$.

In the general case~$0 \le s \le r-1$, we follow the argument in Section~\ref{sec:novelty} with~$V_1:=A \setminus B$ and~$V_2:=B \setminus A$. 
Note that the edges between $V_1$ and $ V_2$, which we denote
$ G_{n,p}[V_1, V_2]$, form a random bipartite graph $ G_{N,N,p}$ with~$N= r-s$.
Of course, we also need to account for the edges between $V:=V_1 \cup V_2 = A \triangle B$ and $ A\cap B$. 
For any vertex~$v \in V= A \triangle B$ we therefore can write
\[
I_v=J_v L_v ,
\]
where~$J_v$ is the indicator random variable for the event that~$v$ is isolated in~$G_{n,p}[V_1,V_2]$, 
and~$L_v$ is the indicator random variable for the event that~$v$ has no neighbors in~$A \cap B = {(A \cup B) \setminus V}$. 
Note that~$L_v$ is a function of the edge-status of all $s=|A \cap B|$ potential pairs~$\{v\} \times (A \cap B)$, 
which in turn implies that all the~$L_v$ variables are independent random variables (since they depend on disjoint sets of random variables). 
Note that~$\sum_{v \in A \triangle B}J_v$ would fit exactly into the framework of Section~\ref{sec:novelty} with~$N=r-s$ and~$V=A \triangle B$. 
The fact that all the~$L_v$ are independent then makes it easy to see that all arguments from Section~\ref{sec:novelty} carry over to~$X=\sum_{v \in V}J_vL_v$ (as all applications of the FKG inequality and all conditional independence arguments remain valid), 
with the following modified parameters 
\begin{equation}\label{eq:mod:param}
\begin{split}
\mu & = 2(r-s)(1-p)^r,\\
\Delta &= 2(r-s)^2(1-p)^{2r-1},\\ 
\Pi_w &= (1-p)^{r-1},\\
\delta_v &=(r-s)(1-p)^{r-1} = \Delta/\mu.
\end{split}
\end{equation}
In particular, the proof of the upper bound in~\eqref{eq:Poisson} again~gives
\begin{equation}
\label{eq:rho:domeachother:Poisson}
\Pr(X=0) \le e^{-\mu + \Delta p \cdot  \log(1+\mu/(\Delta p))} .
\end{equation}

It remains to show that~\eqref{eq:rho:domeachother:Poisson} implies the desired estimate~\eqref{eq:rho:domeachother}. 
Noting that~$\Delta = \Theta(\mu^2)$ as well as~$\mu p \ge (1-p)^r p \gg 1/n$ and~$\mu \le 2r (1-p)^r \le (2+o(1)) \log^3(np)/(np^2)$, 
it follows that
\[
\Delta p \cdot  \log(1+\mu/(\Delta p)) \le O(1) \cdot \frac{\log^6(np)}{n^2p^3} \cdot \log(n) = o(1).
\]
In view of~\eqref{eq:rho:domeachother:rewritten}, \eqref{eq:mod:param} and~\eqref{eq:rho:domeachother:Poisson}, we infer that
\begin{equation}
\label{eq:rho:domeachother:Poisson:upper}
\pr\bigpar{\text{$A$ and $B$ dominate each other}} \le (1+o(1)) e^{-2(r-s)(1-p)^r} .
\end{equation}
Next, note that~\eqref{eq:prsize} implies~${(1-p)^\rsize \ll 1}$, 
and that~\eqref{eq:prsize} combined with the form~\eqref{eq:rsize} of~$\rsize$ yields
\[
(\rsize-s) (1-p)^{2\rsize} \le (1+o(1))\frac{\log^5(np)}{n^2p^3} = o(1).
\] 
Using that the well-known bound~$1+x=e^{x+O(x^2)}$ for~$|x| \le 1/2$, 
it thus readily follows that
\[
(1-2(1-p)^{\rsize})^{\rsize-s}= (1+o(1)) e^{-2(\rsize-s)(1-p)^{\rsize}} ,
\]
which together with~\eqref{eq:rho:domeachother:Poisson:upper} then establishes the desired estimate~\eqref{eq:rho:domeachother}, completing the proof of Lemma~\ref{lem:Poisson}. 
\noproof

\subsubsection{Asymptotic independence: Proof of Lemma~\ref{lem:v1bound}}
\label{sec:v1bound}
In this subsection we prove Lemma~\ref{lem:v1bound}, which concerns intersections of size $0 \leq \sint \leq \rsize_0 = \floor{\rsize^2\log(np)/n}$. 
To establish the claimed bound~\eqref{eq:cancelbound} we need to estimate the probabilities~$\rho(s)/\tau^2$ defined in~\eqref{eq:tau} and~\eqref{eq:rho}.
Combining the improved upper bound~\eqref{eq:rho:improved} on~$\rho(s)$ with the exact form~\eqref{eq:tau} of~$\tau$, 
using the well-known estimate~$1+x \leq e^x$ together with~$(1-p)^r \ll 1$ and~$r \ll n$ it routinely follows that 
\begin{equation}\label{eq:usgeneral}
\begin{split}
\frac{\rho(\sint)}{\tau^2}
&\leq (1+o(1))\lrsqpar{\frac{1-2(1-p)^{\rsize}+(1-p)^{2\rsize-\sint}}{1-2(1-p)^\rsize+(1-p)^{2\rsize}}}^{n-\rsize}\\
& = (1+o(1)) \lrsqpar{1+\frac{(1-p)^{2\rsize-\sint}-(1-p)^{2\rsize}}{1-2(1-p)^\rsize+(1-p)^{2\rsize}}}^{n-r}\\
&\leq (1+o(1)) \exp\Bigcpar{(1+o(1)) n \lrsqpar{(1-p)^{2\rsize-\sint}-(1-p)^{2\rsize}}}.
\end{split}
\end{equation}
Note that the above estimate~\eqref{eq:usgeneral} did not use~$\sint \leq \rsize_0$, i.e., is valid for any~${0 \leq \sint \le \rsize}$. 
By inserting the estimates~\eqref{eq:prsize}--\eqref{eq:binomapprox} and~$\sint \leq \rsize_0 \leq \rsize^2\log(np)/n \ll 1/p$ into~\eqref{eq:usgeneral}, 
using the form~\eqref{eq:rsize} of~$\rsize$ it follows that
\begin{equation}\label{eq:usgeneral:2}
\begin{split}
\frac{\rho(\sint)}{\tau^2} &\leq (1+o(1)) \exp\Bigcpar{(1+o(1))n \cdot \lrsqpar{(1-p)^{-\sint}-1} \cdot (1-p)^{2\rsize} }\\
&\leq (1+o(1))\exp\left\{(1+o(1))n \cdot p\sint \cdot \left(\frac{\log^2(np)}{np}\right)^{2} \right\}\\
&\leq (1+o(1))\exp\left\{(1+o(1)) \frac{\rsize^2\log(np)}{n} \cdot \frac{\log^4(np)}{np} \right\}\\
&\leq (1+o(1))\exp\left\{(1+o(1))\frac{\log^7(np)}{n^2p^3}\right\} = 1+o(1).
\end{split}
\end{equation}
This establishes the desired estimate~\eqref{eq:cancelbound}, completing the proof of Lemma~\ref{lem:v1bound}. 
\noproof

\subsubsection{Counting estimate: Proof of Lemma~\ref{lem:v2bound}}
\label{sec:v2bound}
In this subsection we prove Lemma~\ref{lem:v2bound}, which concerns intersections of size~$\floor{\rsize^2\log(np)/n} = \rsize_0 < \sint \le r$.
To establish the claimed bound~\eqref{eq:usbound} we need to estimate the parameter~$\uu_\sint$ defined in~\eqref{eq:v2}, 
which consists of two ratios: certain binomial coefficients and the probabilities~$\rho(s)/\tau^2$.
For the the probabilities~$\rho(s)/\tau^2$ we will reuse estimate~\eqref{eq:usgeneral}, which is valid for any~$0 \le s \le r$ (as discussed in Section~\ref{sec:v1bound}). 
For the binomial coefficients, we use the standard estimate~$s! \ge (s/e)^s$ together with~$r \ll n$ to obtain that
\begin{equation}\label{eq:usbinom}
\begin{split}
\frac{\binom{n}{\rsize}\binom{\rsize}{\sint}\binom{n-\rsize}{\rsize-\sint}}{\binom{n}{\rsize}\binom{n}{\rsize}} 
& = \frac{1}{s!} \cdot \biggsqpar{\frac{r!}{(r-s)!}}^2 \cdot \frac{(n-r)!}{(n-2r+s)!} \cdot \frac{(n-r)!}{n!} \\
& \le (e/s)^s \cdot r^{2s} \cdot (n-r)^{r-s} \cdot (n-r)^{-r}\\
& \le (e/s)^s \cdot r^{2s} \cdot n^{-s}e^{O(rs/n)} \le \biggpar{\frac{e^{2}r^2}{sn}}^{s} .
\end{split}
\end{equation}
For technical reasons, we now estimate~$\uu_\sint$ using three case distinctions (depending on the range of~$s$).

\vskip3mm

\noindent
\textbf{Case~1: $r_0 < \sint \leq 1/(p\log(np))$.} \\
\noindent
Here we proceed similarly to~\eqref{eq:usgeneral:2} in Section~\ref{sec:v1bound}. Indeed, by inserting the estimates~\eqref{eq:prsize}--\eqref{eq:binomapprox} and~$p\sint \ll 1$ into~\eqref{eq:usgeneral}, 
using the form~\eqref{eq:rsize} of~$\rsize$ it follows (with room to spare) that
\begin{equation}\label{eq:case2}
\begin{split}
\frac{\rho(\sint)}{\tau^2}
& \le \exp\Bigcpar{(1+o(1)) n \cdot \lrsqpar{(1-p)^{-\sint}-1} \cdot (1-p)^{2\rsize}}\\
& \le \exp\Biggcpar{(1+o(1)) n \cdot ps \cdot \left(\frac{\log^2(np)}{np}\right)^{2}}\\
&\leq \exp\Biggcpar{\frac{(1+o(1))\sint \log^4(np)}{np}} \leq e^{o(s)}.
\end{split}
\end{equation}
Combining this estimate with~\eqref{eq:usbinom} and the assumed lower bound~$\sint > r_0$, which implies~$\sint \ge {r_0+1} \ge {\rsize^2\log(np)/n}$, 
it follows that~$\uu_\sint$ defined in~\eqref{eq:v2} satisfies (for all sufficiently large~$n$) the upper bound
\begin{equation}\label{eq:case2:intermediate}
u_s \le \biggpar{\frac{e^{3}r^2}{sn}}^{s} \le \biggpar{\frac{O(1)}{\log(np)}}^{s} \ll \biggpar{\frac{1}{\log(np)}}^{s/2} .
\end{equation}

\vskip3mm

\noindent
\textbf{Case~2: $1/(p\log(np)) < \sint \leq \rsize-1/p$.} \\
\noindent
Proceeding similarly to~\eqref{eq:case2}, 
using~\eqref{eq:usgeneral} and~\eqref{eq:prsize} it follows~that
\begin{equation}\label{eq:case3}
\begin{split}
\frac{\rho(s)}{\tau^2} 
& \le \exp\Bigcpar{(1+o(1)) n \cdot (1-p)^{\rsize} \cdot (1-p)^{\rsize-\sint}}\\
& \le \exp\Bigcpar{(1+o(1)) \log(np) \cdot r (1-p)^{\rsize-\sint}} .
\end{split}
\end{equation}
To estimate~$r (1-p)^{\rsize-\sint}$ we employ another case distinction.
If~$3/4 \cdot \rsize \leq \sint \leq \rsize-1/p$, then 
\[
\frac{r (1-p)^{\rsize-\sint}}{\sint} \le \frac{4}{3} \cdot e^{-p(\rsize-\sint)} \le \frac{4}{3} \cdot e^{-1} \le 0.499 .
\]
If~$1/(p\log(np)) < \sint \leq 3/4 \cdot \rsize$, then 
it follows that
\begin{equation*}
\begin{split}
\frac{r(1-p)^{\rsize-\sint}}{\sint} & \le O(1) \cdot \log^2(np) \cdot (1-p)^{(1-3/4)\rsize} \\
& \le O(1) \cdot \log^2(np) \left(\frac{\log^2(np)}{np}\right)^{1/4} = o(1) .
\end{split}
\end{equation*}
Putting both cases together, in view of~\eqref{eq:case3} we each time obtain that
\begin{equation}\label{eq:case3:2}
\frac{\rho(s)}{\tau^2} \le \exp\Bigcpar{s/2 \cdot  \log(np)} .
\end{equation}
Combining this estimate with~\eqref{eq:usbinom} and the assumed lower bound~$\sint > 1/(p\log(np))$, 
it follows that~$\uu_\sint$ defined in~\eqref{eq:v2} satisfies (for all sufficiently large~$n$) the upper bound
\begin{equation}\label{eq:case3:intermediate}
u_s \le \biggpar{\frac{e^{2}r^2}{sn}}^{s} \cdot (np)^{s/2} \le \biggpar{\frac{O(1) \cdot \log^3(np)}{np}}^{s} \cdot (np)^{s/2} \ll \biggpar{\frac{1}{\log(np)}}^{s/2} .
\end{equation}


\vskip3mm

\noindent
\textbf{Case~3: $\rsize-1/p < \sint \le r$.} \\
\noindent
Here we redo our estimates, in order to exploit that~$r-s$ is small. 
Proceeding similarly to~\eqref{eq:usgeneral} and~\eqref{eq:usbinom}, using~$\E X_\rsize \gg 1$ as well as~$\binom{r}{s}=\binom{r}{r-s} \le r^{r-s}$ and~$r \ll n$ it follows~that
\begin{equation}\label{eq:case5}
\begin{split}
\uu_\sint & \ll \uu_\sint \cdot \E X_\rsize = \binom{\rsize}{\sint}\binom{n-\rsize}{\rsize-\sint} \cdot \frac{\rho(s)}{\tau}\\
& \leq (rn)^{(\rsize-\sint)} \cdot (1+o(1)) \lrsqpar{\frac{1-2(1-p)^{\rsize}+(1-p)^{2\rsize-\sint}}{1-(1-p)^\rsize}}^{n-\rsize} \\
&\leq n^{2(\rsize-\sint)} \cdot (1+o(1)) \lrsqpar{1+\frac{(1-p)^{2\rsize-\sint}-(1-p)^{\rsize}}{1-(1-p)^{\rsize}}}^{n-\rsize} \\
&\leq n^{2(\rsize-\sint)} \cdot (1+o(1)) \exp\Bigcpar{(1+o(1))n(1-p)^\rsize \cdot \bigsqpar{(1-p)^{\rsize-\sint}-1}}.
\end{split}
\end{equation}
It is well-known (and easy to show by induction) that~$(1-x)^t \le 1-tx + \binom{t}{2}x^2$ for any integer~$t \ge 0$ and real~$x \in [0,1]$. 
Since~$0 \le \rsize-\sint \le 1/p$ is an integer, it follows that
\begin{equation*}
(1-p)^{\rsize-\sint} \le 1-p(\rsize-\sint) + [p(\rsize-\sint)]^2/2 \le 1 - p(\rsize-\sint)/2 .
\end{equation*}
Together with~\eqref{eq:case5}, in view of~\eqref{eq:prsize} and~$\log^2(np) \gg \log n$ it follows (for all sufficiently large~$n$) that  
\begin{equation}\label{eq:case5:intermediate}
\begin{split}
\uu_\sint & \ll  \exp\Biggcpar{2(\rsize-\sint)\log n + \frac{(1+o(1))\log^2(np)}{p} \cdot \bigsqpar{-p(\rsize-\sint)/2}} \le \biggpar{\frac{1}{n}}^{(\rsize-\sint)/3} .
\end{split}
\end{equation}

\vskip3mm

To sum up these case distinctions: combining the three above estimates~\eqref{eq:case2:intermediate}, \eqref{eq:case3:intermediate} and~\eqref{eq:case5:intermediate} for~$\uu_\sint$, 
by using that~$\sint > r_0$ implies~$\sint \ge 1$ it readily follows that $\sum_{r_0 < s \le r} \uu_\sint = o(1)$.
This establishes the desired estimate~\eqref{eq:usbound}, completing the proof of Lemma~\ref{lem:v2bound} and thus Theorem~\ref{thm:main}, as discussed. 
\noproof

\section{No two-point concentration: Proof of Lemma~\ref{lem:anti}}\label{sec:ss}
In this section we prove the anti-concentration result Lemma~\ref{lem:anti}, 
which shows that two-point concentration of~$\gamma(G_{n,p})$ fails for~$n^{-1} \ll p \ll (\log n)^{2/3}n^{-2/3}$.  
In fact, we shall prove the following generalization of Lemma~\ref{lem:anti}, 
which intuitively shows that the domination number $\gamma(G_{n,p})$ 
of the binomial random graph~$G_{n,p}$ is not concentrated on~$1+O(\log(np)/(np^{3/2}))$ many~values.
%
\begin{lemma}\label{lem:anti:general}
Fix~$\eps \in (0,1/4]$. Assume that~$p=p(n)$ satisfies~$n^{-1} \ll p \ll 1$.  
Then for an infinite sequence of~$n$ there exists~$q \in \bigl[p(n),2p(n)\bigr]$ such that~$\max_{[s,t]}\Pr\bigl(\gamma(G_{n,q}) \in [s,t]\bigr) \le 1-\eps$, 
where the maximum is taken over any interval~${[s,t]}$ of length~$t-s \le \eps \log(nq)/(nq^{3/2})$. 
\end{lemma}
This result can be proved by an adaptation of the anti-concentration argument of Sah and Sawhney for the independence number, 
which in turn is similar to an argument of Alon and Krivelevich (see Section~5 in~\cite{bohman2022} and Section~4 in~\cite{ak}). 
We shall present our proof as a discrete derivative argument, inspired by the anti-concentration argument of Heckel for the chromatic number (see Section~2.5 in~\cite{h} and Section~2 in~\cite{hr}).
Our proof of Lemma~\ref{lem:anti:general} combines two ingredients: the typical asymptotics of the domination number $\gamma(G_{n,p})$, and a coupling result between~$G_{n,p}$ and~$G_{n,q}$ when~$q=p+O(\sqrt{p}/n)$; see Lemmas~\ref{lem:typical} and~\ref{lem:coupling} below. 
Lemma~\ref{lem:typical} follows from Theorem~1.1 and Observation~2.1 in~\cite{glebov2015}.
Asymptotic variants of Lemma~\ref{lem:coupling} are folklore (see Lemma~3.1 in~\cite{AF} and Example~1.5 in~\cite{SJ}, for example); 
see Appendix~\ref{apx:coupling} for a short non-asymptotic proof.
\begin{lemma}[\cite{glebov2015}]\label{lem:typical}
If~$p=p(n)$ satisfies~$n^{-1} \ll p \ll 1$, then whp $\gamma(G_{n,p}) =  {(1+o(1))\log(np)/p}$. 
\end{lemma}
\begin{lemma}\label{lem:coupling}
For any edge-probabilities~$p=p(n)$ and~$q=q(n)$ with~$p,q \in (0,1)$ there is a coupling of~$G_{n,p}$ and~$G_{n,q}$ such that $\Pr(G_{n,p}=G_{n,q}) \ge 1-n |p-q|/\sqrt{4q(1-q)}$. 
\end{lemma}
\begin{proof}[Proof of Lemma~\ref{lem:anti:general}]
Define~$\ell(n,p):=\eps \log(np)/(np^{3/2})$, $\Delta := 2.2\eps \sqrt{p}/n$, $p_i: = {p + (i-1) \Delta}$ and $I := \floor{p/\Delta}$. 
We henceforth assume that~$n$ is large enough to ensure~${p \le 1/4}$. Note that~$I \ge 1$ and~$p \le p_i \le p_I \le 2p \le 1/2$. 

Aiming at a contradiction, suppose that the conclusion of Lemma~\ref{lem:anti:general} fails.
This implies the following for all sufficiently large~$n$:  
for each~$p_i=p_i(n)$ with~$1 \le i \le I$ there is an interval~${[s_i,t_i]}$ such~that 
\begin{equation}\label{eq:anti:interval}
\Pr(\gamma(G_{n,p_i}) \in [s_i,t_i]) \ge 1-\eps \quad \text{ and } \quad t_i-s_i \le \ell(n,p_i) ,
\end{equation}
where~$s_i$ and~$t_i$ depend on~$n,p_i$ (which we omit in the notation to avoid clutter). 
Since~$p_1 = p$ and $p_I = (1+o(1)) 2p$, the typical asymptotics from Lemma~\ref{lem:typical} imply that~whp 
\begin{equation}\label{eq:typical:average}
\gamma(G_{n,p_1}) - \gamma(G_{n,p_I}) =  (1+o(1))\frac{\log(np)}{p} - (1+o(1))\frac{\log(np)}{2p} =  (1+o(1))\frac{\log(np)}{2p}.
\end{equation}
Invoking the concentration estimate~\eqref{eq:anti:interval}, with probability at least~$1-2\eps$ we also have
\begin{equation}\label{eq:typical:upper}
t_1 - s_I \ge \gamma(G_{n,p_1}) - \gamma(G_{n,p_I}) .
\end{equation}
Using~$p \le p_i \le 1/2$, we see that~${n |p_{i+1}-p_i|/\sqrt{4p_i(1-p_i)}} \le  {n \Delta/\sqrt{2p} \le 1.99 \eps}$ for any~${1 \le i < I}$. 
Combining the coupling from Lemma~\ref{lem:coupling} with the concentration estimate~\eqref{eq:anti:interval},
with probability at least~$1-3.99\eps> 0$ we thus have $\gamma(G_{n,p_i}) \in {[s_i,t_i] \cap [s_{i+1},t_{i+1}]}$, 
which in turn (deterministically) establishes that
\begin{equation}\label{eq:typical:intervals}
s_i \le t_{i+1} \qquad \text{for all~$1 \le i < I$}.
\end{equation}
We now may assume that~\eqref{eq:typical:average} and~\eqref{eq:typical:upper} both hold, 
since for all sufficiently large~$n$ this event happens with probability at least ${1-2\eps-o(1)>0}$. 
Using first~$np \gg 1$ to optimize~$\ell(n,q)=\eps \log(nq)/(nq^{3/2})$, and then~$I \le n \sqrt{p}/(2.2\eps)$, 
by combining~\eqref{eq:anti:interval} with~\eqref{eq:typical:intervals} it follows for all sufficiently large~$n$~that 
\begin{equation*}
\ell(n,p) = \max_{q \in [p,2p]} \ell(n,q) \ge \frac{1}{I} \sum_{1 \le i \le I} \Bigl[t_i-s_i\Bigr] \ge \frac{t_1-s_I}{I} \ge \bigl(1.1+o(1)\bigr)\frac{\eps\log(np)}{np^{3/2}} ,
\end{equation*}
which gives a contradiction to the definition of~$\ell(n,p)$ for all sufficiently large~$n$. 
\end{proof}

\section{Conclusion and open questions}\label{sec:conclusion}
Our main result Theorem~\ref{thm:main} gives a detailed account of the concentration of the domination number~$\gamma(G_{n,p})$ of~$G_{n,p}$ for $p \ge (\log n)^3 n^{-2/3}$, 
but our understanding of the concentration of~$\gamma(G_{n,p})$ due to Lemma~\ref{lem:anti:general} remains somewhat unsatisfactory for $n^{-1}  \ll p \ll n^{-2/3}$.
%
Given this situation, there are a number of natural and interesting open questions. 
\begin{itemize}
    \item{\bf Concentration.} 
It would be interesting to determine the extent of concentration of~$\gamma(G_{n,p})$ for $n^{-1} \ll p \ll n^{-2/3}$,
where we believe that Lemma~\ref{lem:anti:general} predicts the correct answer.  
    \begin{quest}
    \label{quest:concentrate}
Suppose~$p=p(n)$ satisfies~$n^{-1} \ll p \ll n^{-2/3}$. 
Does there exist an interval~$ I = I(n)$ of length $1/ (np^{3/2}) $ times some poly-logarithmic factor such that
        \[ \Pr ( \gamma(G_{n,p}) \in I ) \to 1 ?\]
    \end{quest}

For~$n^{-1} \ll p \ll n^{-2/3}$, a closely related problem is to determine the precise location of the shortest interval~$I$ on which~$\gamma( G_{n,p})$ is concentrated.  
Interestingly, Glebov, Liebenau and Szab\'o proved (using a clever sprinkling argument) that this location must be some distance from the expectation threshold~$ \hat{r}$ used~\eqref{def:hatr} in this paper:
indeed, Theorem~1.4(c) in \cite{glebov2015} implies that, for some constant~$c>0$, 
	\begin{equation} \label{eq:gls} 
	\Pr\left( \gamma(G_{n,p}) \le \hat{r} + \frac{c\log(np)}{n p^{3/2}} \right) \not\to 1.
	\end{equation}

    \item {\bf Anti-concentration.} 
It would be desirable to prove a stronger anti-concentration result for~$\gamma(G_{n,p})$, 
in particular to obtain a statement that applies to all edge-probabilities~$p=p(n)$ with~$n^{-1} \ll p \ll n^{-2/3}$. 
\begin{quest}\label{quest:anticonc}
Suppose~$p=p(n)$ satisfies~$n^{-1} \ll p \ll n^{-2/3}$. 
Are there fixed~${\eps_0,n_0>0}$ such that  
\[
\max_{n \ge n_0}\max_{I}\Pr\bigl(\gamma(G_{n,p}) \in I\bigr) \le 1-\eps_0,
\]
where the maximum is taken over any interval~$I$ of length $ 1/ (np^{3/2}) $ times some poly-logarithmic factor?
\end{quest}	
An upper bound on the probability that $\gamma(G_{n,p})$ is a particular integer could yield an even stronger anti-concentration statement. 
Indeed, one could answer Question~\ref{quest:anticonc} affirmatively by showing~that 
\[
\max_{k \ge 0}\Pr(\gamma(G_{n,p}) = k) = \frac{\tilde{O}(1)}{n p^{3/2}},
\]
where~$\tilde{O}$ suppresses logarithmic factors, as usual.  
This anti-concentration connection makes it an interesting open problem to establish good upper bounds on the point probabilities~$\Pr(\gamma(G_{n,p}) = k)$.

\item{{\bf Domination number of \bf $ G_{n,m}$}.} 
While two-point concentration of $ \gamma ( G_{n,m} )$ for the uniform random graph with $ m \gg (\log n)^3 n^{4/3} $ edges follows from Theorem~\ref{thm:main}, 
the anti-concentration argument that establishes Lemma~\ref{lem:anti:general} does not readily adapt to $G_{n,m}$. 
Hence the problem of determining the extent of concentration of $ \gamma ( G_{n,m})$ for $n \ll m \ll n^{4/3} $ is widely open. 
Interestingly, in this regime it may be possible that $\gamma (G_{n,m})$ with ${m = p \binom{n}{2}}$ is concentrated on a shorter interval than $\gamma ( G_{n,p})$. 
If this is the case, then variations in~$\gamma( G_{n,p})$ could largely be a result of variation in the number of edges in~$G_{n,p}$. 
%
%
\end{itemize}


\bigskip{\noindent\bf Acknowledgements.} 
We thank Svante Janson for useful comments on Section~\ref{sec:novelty}.

\footnotesize
\bibliographystyle{plain}

\begin{thebibliography}{10}

\bibitem{AF}
D.~Achlioptas and E.~Friedgut.
\newblock A sharp threshold for {$k$}-colorability.
\newblock {\em Random Structures \& Algorithms} {\bf 14} (1999) 63--70.

\bibitem{AN}
D.~Achlioptas and A.~Naor. 
\newblock The two possible values of the chromatic number of a random graph.
\newblock {\em Annals of Mathematics} {\bf 162} (2005), 1335--1351.

\bibitem{ABFKR2002}
J.~Alber, H.L.~Bodlaender, H.~Fernau, T.~Kloks, and R.~Niedermeier.
\newblock Fixed parameter algorithms for dominating set and related problems on planar graphs.
\newblock {\em Algorithmica} {\bf 33} (2002), 461--493.

\bibitem{AFR2004}
J.~Alber, M.R.~Fellows, and R.Niedermeier.
\newblock Polynomial-time data reduction for dominating set.
\newblock {\em Journal of the ACM} {\bf 51} (2004), 363--384.

\bibitem{ak}
N.~Alon and M.~Krivelevich.
\newblock The concentration of the chromatic number of random graphs.
\newblock {\em Combinatorica} {\bf 17} (1997), 303--313.

\bibitem{PM}
N.~Alon and J.~Spencer.
\newblock {\em The Probabilistic Method}, 4th~ed.
\newblock Wiley-Interscience, New York (2016).

\bibitem{bohman2022}
T.~Bohman and J.~Hofstad.
\newblock Two-point concentration of the independence number of the random   graph.
\newblock {\em Forum of Mathematics, Sigma}, to appear (2022). \href{https://arxiv.org/abs/2208.00117}{\texttt{arXiv:2208.00117}}

\bibitem{BC1979}
B.~Bollob\'{a}s and E.J.~Cockayne, 
\newblock Graph-theoretic parameters concerning domination, independence, and irredundance.
\newblock {\em Journal of Graph Theory} {\bf 3} (1997), 241--249.

\bibitem{1976cliques}
B.~Bollob{\'a}s and P.~Erd\H{o}s.
\newblock Cliques in random graphs.
\newblock {\em Mathematical Proceedings of the Cambridge Philosophical Society} {\bf 80} (1976), 419--427.

\bibitem{B2001}
B.~Bollob{\'a}s,
\newblock {\em Random Graphs}, 2nd ed.
\newblock Cambridge University Press, Cambridge (2001). 

\bibitem{DFHT2005}
E.D.~Demaine, F.V.~Fomin, M.~Hajiaghayi, and D.M.~Thilikos.
\newblock Subexponential parameterized algorithms on bounded-genus graphs and {$H$}-minor-free graphs.
\newblock {\em Journal of the ACM} {\bf 52} (2005), 866--893.

\bibitem{FK2016}
A.~Frieze and M.~Karo{\'n}ski.
\newblock {\em Introduction to Random Graphs}.
\newblock Cambridge University Press (2016).

\bibitem{GJ1979}
M.R.~Garey and D.S.~Johnson.
\newblock {\em Computers and intractability: A Guide to the Theory of NP-Completeness}.
\newblock W.H.~Freeman, San Francisco, CA (1979).

\bibitem{glebov2015}
R.~Glebov, A.~Liebenau, and T.~Szab\'{o}.
\newblock On the concentration of the domination number of the random graph.
\newblock {\em SIAM Journal on Discrete Mathematics} {\bf 29} (2015), 1186--1206. 

\bibitem{HHSS1998}
T.W.~Haynes, S.T.~Hedetniemi, and P.~J.~Slater.
\newblock {\em Fundamentals of domination in graphs}.
\newblock Marcel Dekker, Inc., New York (1998).

\bibitem{HHHH2002}
T.W.~Haynes, S.M.~Hedetniemi, S.T.~Hedetniemi, and M.A.~Henning.
\newblock Domination in graphs applied to electric power networks.
\newblock {\em SIAM Journal on Discrete Mathematics} {\bf 15} (2002), 519--529.

\bibitem{h}
A.~Heckel.
\newblock Non-concentration of the chromatic number of a random graph.
\newblock {\em Journal of the American Mathematical Society} {\bf 34} (2021), 245--260.

\bibitem{hr}
A.~Heckel and O.~Riordan.
\newblock How does the chromatic number of a random graph vary?
\newblock {\em Journal of the London Mathematical Society} {\bf 108} (2023), 1769--1815.

\bibitem{SJ1990}
S.~Janson.
\newblock Poisson approximation for large deviations.
\newblock {\em Random Structures \& Algorithms} {\bf 1} (1990), 221--229.

\bibitem{SJ}
S.~Janson.
\newblock Asymptotic equivalence and contiguity of some random graphs.
\newblock {\em Random Structures \& Algorithms} {\bf 36} (2010), 26--45. 

\bibitem{JLR1990}
S.~Janson, T.~{\L}uczak, and A.~Ruci\'{n}ski.
\newblock An exponential bound for the probability of nonexistence of a specified subgraph in a random graph.
\newblock In {\em Random graphs '87 ({P}ozna\'n, 1987)}, pp.~73--87, Wiley, Chichester (1990).

\bibitem{JLR2000}
S.~Janson, T.~{\L}uczak, and A.~Ruci{\'n}ski.
\newblock {\em Random Graphs}.
\newblock Wiley-Interscience, New York (2000).

\bibitem{JS2016}
S.~Janson and L.~Warnke.
\newblock The lower tail: {P}oisson approximation revisited.
\newblock {\em Random Structures \& Algorithms} {\bf 48} (2016), 219--246. 

\bibitem{L1991}
T.~{\L}uczak.
\newblock A note on the sharp concentration of the chromatic number of random graphs.
\newblock {\em Combinatorica} {\bf 11} (1991), 295--297.

\bibitem{matula}
D.~Matula.
\newblock On complete subgraphs of a random graph.
\newblock In {\em Proceedings of the Second Chapel Hill Conference on Combinatory Mathematics and its Applications}, pp.~356--369, University of North Carolina at Chapel Hill (1970).

\bibitem{NS1993}
S.E.~Nikoletseas and P.G.~Spirakis.
\newblock Near-optimal dominating sets in dense random graphs in polynomial expected time.
\newblock In {\em Graph-theoretic concepts in computer science ({U}trecht, 1993)}, pp.~1--10, Springer, Berlin (1994).

\bibitem{Ore1962}
O.~Ore.
\newblock {\em Theory of graphs}.
\newblock American Mathematical Society, Providence, RI (1962). 

\bibitem{RW2015}
O.~Riordan and L.~Warnke.
\newblock The {J}anson inequalities for general up-sets.
\newblock {\em Random Structures \& Algorithms} {\bf 46} (2015), 391--395.

\bibitem{ss}
E.~Shamir and J.~Spencer.
\newblock Sharp concentration of the chromatic number on random graphs~$G_{n,p}$.
\newblock {\em Combinatorica} {\bf 7} (1987), 121--129.

\bibitem{SW2022}
E.~Surya and L.~Warnke.
\newblock On the concentration of the chromatic number of random graphs.
\newblock {\em Electronic Journal of Combinatorics}, to appear (2022). \href{https://arxiv.org/abs/2201.00906}{\texttt{arXiv:2201.00906}}

\bibitem{SWZ2023}
E.~Surya, L.~Warnke, and E.~Zhu.
\newblock Isomorphisms between dense random graphs.
\newblock Preprint (2023). \href{https://arxiv.org/abs/2305.04850}{\texttt{arXiv:2305.04850}}

\bibitem{W1981}
K.~Weber. 
\newblock Domination number for almost every graph.
\newblock {\em Rostocker Mathematisches Kolloquium} {\bf 16} (1981), 31--43.

\bibitem{West1996}
D.~West.
\newblock {\em Introduction to graph theory}.
\newblock Prentice Hall, Upper Saddle River, NJ (1996). 

\bibitem{WG2001}
B.~Wieland and A.P.~Godbole.
\newblock On the domination number of a random graph.
\newblock {\em Electronic Journal of Combinatorics} {\bf 8} (2001), Paper 37.

\end{thebibliography}

\normalsize

\begin{appendix}

\section{Appendix: Proof of the coupling result Lemma~\ref{lem:coupling}}\label{apx:coupling}
\begin{proof}[Proof of Lemma~\ref{lem:coupling}]
Note that~$G_{n,p}$ conditioned on having~$m$ edges has the uniform distribution. 
Denoting the number of edges of the two random graphs by~$X:=e(G_{n,p})$ and~$Y:=e(G_{n,q})$, it follows that the total variation distance~satisfies
\[
d_{\mathrm{TV}}(G_{n,p},G_{n,q}) \le d_{\mathrm{TV}}(X,Y). 
\]
Let~$P$ and~$Q$ denote independent Bernoulli random variables with success probability~$p$ and~$q$, respectively.
Using first Pinsker's inequality and then additivity of the  Kullback–Leibler divergence for product distributions, it follows that 
\[
d_{\mathrm{TV}}(X,Y) \leq \sqrt {{\frac{1}{2}}D_{\mathrm{KL}}(X\parallel Y)} = \sqrt {{\frac{1}{2}} \cdot \binom{n}{2} \cdot D_{\mathrm{KL}}(P\parallel Q)} .
\]
Using the well-known inequality~$\log(1+x) \le x$, a direct calculation then shows that
\begin{align*}
D_{\mathrm{KL}}(P\parallel Q)  & =  p \log\Bigl(\frac{p}{q}\Bigr) + (1-p) \log\Bigl(\frac{1-p}{1-q}\Bigr) \\
& \le \frac{p(p-q)}{q} + \frac{(1-p)(q-p)}{(1-q)} = \frac{(p-q)^2}{q(1-q)} ,
\end{align*}
which in view of~$\binom{n}{2} \leq n^2/2$ completes the proof of Lemma~\ref{lem:coupling}. 
\end{proof}

\end{appendix}

\end{document}